\documentclass{article}

\usepackage{amsfonts}
\usepackage{amsmath}
\usepackage{amssymb}
\usepackage{amsthm}
\usepackage{fontenc}
\usepackage{mathrsfs}
\usepackage{graphicx}
\usepackage{verbatim}
\usepackage{dsfont}
\usepackage{bbm}

\newcommand{\ud}{\,\mathrm{d}}

\newcommand{\reals}{\mathbb{R}}
\newcommand{\curlyB}{\mathscr{B}}
\newcommand{\gL}{\gamma(L^2(0,t),E_\eta)}

\newcommand{\bn}{\big\|}
\newcommand{\V}{V^p_{\alpha,\infty}\big([0,T]\times\Omega;E_\eta\big)}

\newcommand{\subV}{V^p_{\alpha,\infty}([0,T]\times\Omega;E_\eta)}

\numberwithin{equation}{section}

\theoremstyle{definition}
\newtheorem{defn}{Definition}[section]

\theoremstyle{plain}
\newtheorem{thm}[defn]{Theorem}

\theoremstyle{plain}
\newtheorem{prop}[defn]{Proposition}

\theoremstyle{plain}
\newtheorem{lemma}[defn]{Lemma}

\theoremstyle{plain}
\newtheorem{cor}[defn]{Corollary}

\theoremstyle{definition}
\newtheorem{expl}[defn]{Example}

\theoremstyle{definition}

\theoremstyle{remark}

\title{Infinitely delayed stochastic evolution equations in UMD Banach spaces}
\author{P. ~Crewe}

\begin{document}

\maketitle

\begin{abstract}
  We prove an existence and uniqueness result for the infinitely delayed stochastic evolution equation
  \begin{equation}
\left\{
\begin{array}{rll}
  \ud U(t) &= &\big(AU(t) + F(t,U_t)\big) \ud t + B(t,U_t)\ud W_H(t),\;\; t\in[0,T_0]\\
  U_0&: &\Omega \to \mathscr{B},
\end{array} \right.\nonumber
\end{equation}
where $A$ is the generator of an analytic semigroup on a UMD space $E$, $F$ and $B$ satisfy Lipschitz conditions and $\mathscr{B}$ is a weighted $L^p$ history space. This paper is based on recent work of van Neerven \emph{et al.}~which developed the theory of abstract stochastic evolution equations in UMD spaces.
\end{abstract}


\section{Introduction}
In this paper we prove a mild existence and uniqueness result for a class of infinitely delayed abstract stochastic evolution equations in UMD Banach spaces. Delay equations have been long studied and have applications in a wide range of fields from biology to material science. Finite and infinite delay problems are different in flavour, the former being far more amenable to traditional semigroup methods of analysis, but both have a rich literature. Stochastic differential equations are much used in finance but due to recent advances are increasingly being applied far more widely.

We study the equation
\begin{equation}
\label{eqn:intro_DSDE}
\left\{
\begin{array}{rll}
  \ud U(t) &= &\big(AU(t) + F(t,U_t)\big) \ud t + B(t,U_t)\ud W_H(t),\;\; t\in[0,T_0]\\
  U_0&: &\Omega \to \mathscr{B},
\end{array} \right.
\end{equation}
where $U_t$ is the history function on $(-\infty,0]$ defined by $U_t(s) = U(t+s)$, $A$ is the generator of an analytic $C_0$-semigroup $S$ on a UMD space $E$ and $W_H$ is an $H$-cylindrical Brownian motion. The space $\curlyB$ of initial data is of a type introduced by Hale and Kato in \cite{halekato} and $F$ and $B$, mapping $[0,T_0]\times\Omega\times\curlyB$ to $E$ and $\mathcal{L}(H,E)$ respectively, are assumed to satisfy certain Lipschitz conditions.

This work is based on the recent results of van Neerven, Veraar and Weis \cite{VanN-Weis_1}, in which fixed point methods are used to prove existence and uniqueness for the non-delayed problem
\begin{equation}
\label{eqn:intro_SDE}
\left\{
\begin{array}{rll}
  \ud U(t) &= &\big(AU(t) + F(t,U(t))\big) \ud t + B(t,U(t))\ud W_H(t),\;\; t\in[0,T_0] \\
  U(0)&= &u_0
\end{array} \right.
\end{equation}
in a UMD space. An appropriate space $V$ of processes on $[0,T_0]$ is constructed in such a way that if $U$ is in $V$, then both the deterministic and stochastic convolutions
\[
 \int_0^t S(t-s)F(s,U(s))\ud s \quad\text{and}\quad \int_0^t S(t-s)B(s,U(s))\ud W_H(s)
\]
also lie in $V$. The difficulty arises (by a result from \cite{VanN-Veraar}) from the fact that, roughly speaking, if $f(u)$ is stochastically integrable for every $E$-valued stochastically integrable function $u$ and every Lipschitz function $f:E\to E$ then $E$ is isomorphic to a Hilbert space. Van Neerven \emph{et al.}\ overcome this problem by replacing the standard Lipschitz condition with a stronger notion, $L^2_\gamma$-Lipschitz, essentially a Gaussian version of the former, reducing the problem to one of finding a space $V$ such that if $\phi\in V$ then the pathwise convolutions
\begin{equation}
 \label{eqn:intro_convolutions}
 t\mapsto\int_0^t S(t-s)\phi(s)\ud s \quad\text{and}\quad t\mapsto\int_0^t S(t-s)\phi(s)\ud W_H(s)
\end{equation}
also define processes in $V$. Estimates for these integrals are then found through extensive use of the $\gamma$-boundedness of certain families of operators associated with the analytic $C_0$-semigroup $S$.

A theory of stochastic evolution equations of the type \eqref{eqn:intro_SDE} in Hilbert spaces has been in development for over twenty years, for example by Da Prato and Zabczyk \cite{DaPrato-Zabczyk} and much of this work was then extended to spaces of martingale type-$2$ \cite{Brzez_1,Brzez_2}. Our main reference \cite{VanN-Weis_1} of van Neerven \emph{et al.}~is an application to evolution equations of work by the same authors \cite{VanN-Weis_2}, in which they construct a new theory of stochastic integration in general UMD spaces, complete with a double sided It\^o inequality.

Stochastic delay equations have also been recently studied, for example by Cox and G\'orajski \cite{Sonja} who consider the finite delay problem in Banach spaces of type 2 by proving the equivalence of solutions to the stochastic delay equation and the associated stochastic Cauchy problem. Riedle \cite{Riedle} and van Neerven and Riedle \cite{VanN-Riedle} consider a delayed problem with additive noise in $\reals$ and $\reals^n$ with history data from a range of spaces which includes the space $\curlyB$ used in this article. Solutions are represented in the second dual of the history space. Existence of invariant measures for delayed stochastic systems is studied in \cite{vGaans} by van Gaans \emph{et al.}.

Infinitely delayed (deterministic) Cauchy problems have also been much studied. Unlike in problems with bounded (finite) delay, the system can never `forget' the initial data, and therefore the choice of history space $\curlyB$ is important for well-posedness. An axiomatic framework for the space of initial history functions is due to Hale and Kato \cite{halekato}, but most modern work on the subject uses the notation of Hino from \cite{hino}. There is work in the literature about infinitely delayed stochastic problems, for example \cite{liu_2, liu_1}, but all, as far as we can tell, is set in Hilbert spaces.

In this paper we first introduce briefly the tools used in the sequel, namely the stochastic integral in UMD spaces and notions related to $\gamma$ boundedness for families of operators. In Section \ref{sec:Results} we present out main result, the well posedness of \eqref{eqn:intro_DSDE} under suitable conditions and then in Section \ref{sec:example} we give an example based on a stochastic heat equation in a material with memory.

\section{Preliminaries}
\label{sec:prelims}

This section contains various definitions and results (without proof) that will be used below. Much of the following appears largely as it does in \cite{VanN-Weis_1}, and references to the particular results are given. 

Throughout, we will use $(\Omega,\mathscr{F},\mathbb{P})$ to denote a complete probability space with a filtration $\big\{\mathscr{F}_t\big\}_{t\geq0}$. For a finite measure space $(S,\Sigma,\mu)$ and a Banach space $E$, $L^0(S,E)$ denotes the space of equivalence classes of strongly measurable functions from $S$ to $E$. We will often work on sub-intervals $[0,T] \subseteq [0,T_0]$ and will need to keep track of the dependence of various constants upon $T$. We use $C$ for generic constants which may depend on $T_0$ but not on $T$. The value of $C$ may vary from line to line. 

All vector spaces are real valued.

\subsection{Deterministic delayed Cauchy problems}
\label{subsec:deterministicDCPs}

First we present an example of a history space $\curlyB$ of functions $(-\infty,0]\to E$ with the property that if the initial history of the deterministic delay problem
\begin{equation}
\label{eqn:deterministicDDE}
\left\{
\begin{array}{rll}
  u'(t)&= &Au(t) + F(t,u_t) \qquad t\in[0,T]\\
  u_0&= &\phi \in \mathscr{B}
\end{array} \right.
\end{equation}
is in $\curlyB$ and $F : [0,T]\times\curlyB\to E$ is uniformly Lipschitz then \eqref{eqn:deterministicDDE} is well posed. Here $A$ is assumed to be the generator of a $C_0$-semigroup, the history function $u_t:(-\infty,T]\to E$ is defined by $u_t(s) = u(t+s)$ and.

Abstract history spaces with this property were developed by Hale and Kato \cite{halekato}, however we do not require full generality and so give only a relevant example of Hino \cite{hino}.

\begin{expl}
\label{ex:weightedhistoryspace}
For a function $g:(-\infty,0]\to(0,\infty)$ such that
\begin{enumerate}
\item $g$ is Lebesgue integrable on $[-r,0]$ for any $r\leq 0$;
\item there exists a non-negative and locally bounded function $G$ on $(-\infty,0]$ such that $g(s + \theta)\leq G(s)g(\theta)$ for all $s\leq0$ and almost all $\theta\leq 0$,
\end{enumerate}
define $\mathscr{B} := L_g^p(-\infty,0;E)\times E$, where $L_g^p(-\infty,0;E) :=L^p(-\infty,0,g(t)\ud t; E)$ is the weighted Lebesgue space of all classes of functions $\phi:(-\infty,0]\to E$ with norm
\[\|\phi\|^p_{g} := \int_{-\infty}^{0}g(\theta)\|\phi(\theta)\|^p\ud \theta < \infty.\]
Then $\mathscr{B}$ with the norm $\|(\phi,x)\|_{\curlyB} := \|x\| + \|\phi\|_g$ is a Banach space for $p\in[1,\infty)$ and satisfies the axioms of Hino \cite{hino}.
\end{expl}

\begin{lemma}[$\curlyB$ satisfies the axioms of Hino]
\label{lemma:def:phasespace}
The space $\curlyB$ from Example \ref{ex:weightedhistoryspace} has the following properties.
 
Suppose $\phi:(-\infty, T]\to E$ is continuous on $[0,T]$ and $(\phi_0, \phi(0))\in\curlyB$, then $(\phi_t,\phi(t))\in \curlyB$ for each $t\in[0,T]$, and moreover, the map $t\mapsto (\phi_t,\phi(t))$ is continuous from $[0,T]$ to $(\curlyB,\|\cdot\|_{\curlyB})$.

Defining the functions $K$ and $M$ by 
\begin{align}
  K(t) &= 1+\Big(\int_{-t}^0g(u)\ud u\Big)^{\frac{1}{p}}\nonumber\\
  M(t) &= \max\Big\{\Big(\int_{-t}^0g(u)\ud u\Big)^{\frac{1}{p}}, G(-t)^{\frac{1}{p}}\Big\},\nonumber
\end{align}
 we have 
\begin{equation}
 \|\phi(t)\| \leq \big\|\big(\phi_t,\phi(t)\big)\big\|_\mathscr{B} \leq K(t)\sup_{s\in[0,t]}\|\phi(s)\| + M(t)\big\|\big(\phi_0,\phi(0)\big)\big\|_\mathscr{B}\label{eqn:KM_ineq}
\end{equation}
for each $t\in[0,T]$. Here $K$ is continuous and $M$ is locally bounded.
\end{lemma}

We will adopt the convention that elements $(\phi,\phi(0))$ of $\curlyB$ are interpreted as functions $\phi\in L^p_g\big(-\infty,0;E\big)$ which take a value $\phi(0)$ at $0$, and unless otherwise stated we refer simply to $\phi \in \curlyB$. In all relevant cases, when given $\phi$ on $(-\infty,T]$ we will have $\phi$ continuous on $[0,T]$ and so the pair $(\phi_t,\phi(t))$ is well defined for $t\in[0,T]$.

For the generator $A$ of a $C_0$-semigroup $S$ on $E$, choose a number $w\in\reals$ such that $(A-w)$ generates a uniformly exponentially stable semigroup. The fractional powers $(A-w)^\eta$ are well defined and for $\eta>0$ the interpolation space 
\[
E_\eta := \mathcal{D}\big((A-w)^\eta\big).
\]
becomes a Banach space when endowed with the norm
\[
\|x\|_{E_\eta} := \|x\| + \big\|(A-w)^\eta x\big\|.
\]
Up to equivalence of norms this definition is independent of the choice of large enough $w\in\reals$. 

From now on we will assume that $A$ generates an analytic $C_0$-semigroup and will make frequent use of the standard estimate
\begin{equation}
\label{eqn:analyticproperty}
  \|S(t)\|_{\mathcal{L}(E,E_\eta)} \leq C_\eta t^{-\eta}
\end{equation}
for $t\in(0,T_0]$, $\eta>0$. For a history space $\curlyB$ as in Example \ref{ex:weightedhistoryspace} we define
\[
  \mathscr{B}_\eta := \{(A-w)^{-\eta}\phi : \phi \in \mathscr{B}\}
\]
to be the set of functions $\psi:(-\infty,0]\to E_\eta$ such that $(A-w)^{\eta}\psi\in\mathscr{B}$, with norm
\begin{equation}
\label{eqn:defbetanorm}
\|\psi\|_{\mathscr{B}_\eta} := \|(A-w)^\eta\psi\|_{\mathscr{B}}.
\end{equation}

\begin{lemma}
If $\curlyB = L^p_g\big((-\infty,0];E\big)\times E$ as in Example \ref{ex:weightedhistoryspace}, then the spaces $\curlyB_\eta$ and $L^p_g\big((-\infty,0];E_\eta\big)\times E_\eta$ are the same up to equivalent norms.
\end{lemma}

\subsection{Stochastic integration and $\gamma$-boundedness}

In this section we recall the tools we will need to estimate the convolutions in \eqref{eqn:intro_convolutions}, namely the notions of $\gamma$-radonifying for individual operators and $\gamma$-boundedness for families of operators. We also briefly construct the stochastic integral of van Neerven \emph{et al.}. For more details we recommend the paper \cite{VanN-Weis_2}.

\begin{defn}
\label{def:gamma-radonifying}
A bounded linear operator $R\in \mathcal{L}(H, E)$ from a separable Hilbert space $H$ into a Banach space $E$ is said to be \emph{$\gamma$-Radonifying} if for some (and hence every) orthonormal basis $(h_n)_{n\in\mathbb{N}}$ of $H$, the Gaussian sum
\[
\sum_{n\in\mathbb{N}}\gamma_nRh_n
\]
converges in $L^2(\Omega;E)$. Here $(\gamma_n)_{n\in\mathbb{N}}$ is a \emph{Gaussian sequence}, a sequence of independent standard real Gaussian random variables. The space $\gamma(H,E)$ of all $\gamma$-Radonifying operators from $H$ into $E$ becomes a Banach space with respect to the norm
\begin{equation}
\label{eqn:gamma-norm}
  \|R\|_{\gamma(H,E)} := \Bigg(\mathbb{E}\Big\|\sum_{n\in\mathbb{N}}\gamma_nRh_n\Big\|^2\Bigg)^{\frac{1}{2}}.
\end{equation}
This norm is independent of the choice of orthonormal basis for $H$.
\end{defn}

The next result is known as the $\gamma$-Fubini isomorphism, a tool we will use repeatedly in estimating $\gamma$-norms.

\begin{prop}[Proposition 2.6 of \cite{VanN-Weis_2}]
\label{prop:2.6}
Let $(S,\Sigma,\mu)$ be a $\sigma$-finite measure space and let $p\in [1,\infty)$, then the map $F_\gamma: L^p\big(S;\gamma(H,E)\big) \to \mathcal{L}\big(H,L^p(S;E)\big)$ defined by
\[
\big(F_\gamma(\phi)h\big)(s) := \phi(s)h, \quad s\in S, h\in H
\]
defines an isomorphism from $L^p\big(S;\gamma(H,E)\big)$ \emph{onto} $\gamma\big(H,L^p(S;E)\big)$.
\end{prop}

In the delayed context, we will also need the following immediate corollary.

\begin{cor}
\label{cor:product_gamma_fub}
 If $(S,\Sigma,\mu)$ is a $\sigma$-finite measure space and $p\in [1,\infty)$, then the map $G_\gamma: L^p\big(S;\gamma(H,E)\big)\times \gamma(H,E) \to \mathcal{L}\big(H,L^p(S;E)\times E\big)$ defined by
\[
 G_\gamma\big(\phi,R\big)h := \big(\phi(\cdot)h,Rh\big)
\]
is an isomorphism from $L^p\big(S;\gamma(H,E)\big)\times\gamma(H,E)$ \emph{onto} $\gamma\big(H,L^p(S;E)\times E\big)$.
\end{cor}

In order to construct the stochastic integral that will be used throughout this paper, we start by giving the definition of operator valued Brownian motion. For a Hilbert space $\mathscr{H}$, an \emph{$\mathscr{H}$-isonormal process} on a probability space $(\Omega,\mathscr{F},\mathbb{P})$ is a mapping $\mathscr{H}\to L^2(\Omega)$ such that
\begin{itemize}
  \item[(i)] $W(h)$ is a centred Gaussian random variable for each $h\in\mathscr{H}$;
  \item[(ii)] $\mathbb{E}\big(W(h)W(g)\big) = \langle h,g\rangle_{\mathscr{H}}$ for each $h,g\in\mathscr{H}$.
\end{itemize}

Now consider the case $\mathscr{H} := L^2(\reals_+;H)$. An $L^2(\reals_+;H)$-isonormal process induces a family of operators $H\to L^2(\Omega)$ by
\[
  W_H(t)h := W\big(\mathbbm{1}_{(0,t]}\otimes h\big). 
\]
We call $W_H := \big(W_H(t)\big)_{t\in[0,T]}$ an \emph{$H$-cylindrical Brownian motion}. $W_H(t)$ is linear for each $t\in[0,T]$ by (ii) and each $\big(W_H(t)h\big)_{t\in[0,T]}$ is a real-valued Brownian motion.

The natural filtration on $\Omega$ generated by $W_H$ will be denoted $\big\{\mathscr{F}_t\big\}_{t\in[0,T]}$. For the indicator process $\mathbbm{1}_{(a,b]\times B}\otimes(h\otimes x):(0,T)\times\Omega\to\mathcal{L}(H,E)$ defined for $0\leq a < b < T$ and an $\mathscr{F}_a$-measurable subset $B$ of $\Omega$ by 
\[
 \mathbbm{1}_{(a,b]\times B}\otimes(h\otimes x)(t,\omega)g := \mathbbm{1}_{(a,b]\times B}(t,\omega)\langle h,g \rangle_H x
\]
we define the \emph{stochastic integral with respect to $W_H$} as
\[
 \int_0^T \mathbbm{1}_{(a,b]\times B}\otimes(h\otimes x) \ud W_H := \mathbbm{1}_B\big(W_H(b)h - W_H(a)h\big)x.
\]
We extend this definition to adapted step processes $\Phi:(0,T)\times\Omega\to\mathcal{L}(H,E)$ with values in the finite rank operators in the usual way by linearity.

We say a process $\Phi:(0,T)\times\Omega\to \mathcal{L}(H,E)$ is \emph{$H$-strongly measurable} if $\Phi h$ is strongly measurable for each $h\in H$ (where $(\Phi h)(t,\omega) = \Phi(t,\omega)h$).

\begin{defn}
\label{def:stoch_integrable}
 An $H$-strongly measurable process $\Phi$ is \emph{stochastically integrable} with respect to an $H$-cylindrical Brownian motion $W_H$ if there exists a sequence $(\Phi_n)$ of adapted step processes $\Phi_n:[0,T]\times\Omega\to\mathcal{L}(H,E)$ taking values in the finite rank operators and a pathwise continuous process $\xi:[0,T]\times\Omega\to E$ such that 
\begin{itemize}
\item[(i)]$\lim_{n\to\infty}\Phi_nh = \Phi h$ in $L^0\big((0,T)\times\Omega;E\big)$ for all $h\in H$;
\item[(ii)]$\lim_{n\to\infty}\int_0^\cdot\Phi_n \ud W_H = \xi$ in $L^0\big(\Omega;C([0,T];E)\big)$. 
\end{itemize}
Then $\xi$ is a uniquely determined element of $L^0\big(\Omega;C([0,T];E)\big)$, justifying the notation
\[
\xi := \int_0^\cdot \Phi \ud W_H.
\]
The process $\xi$ is then a continuous local Martingale starting at zero and is called the \emph{stochastic integral} of $\Phi$ with respect to $W_H$.
\end{defn}

A Banach space $E$ is said to be a \emph{UMD space} if for some (and hence all) $p\in(1,\infty)$ there exists a constant $\beta_{p,E} \geq 1$ such that for $n\geq1$ every martingale sequence $(d_j)^n_{j=1}$ in $L^p(\Omega;E)$ and every $\{-1,1\}$-valued sequence $(\varepsilon_j)^n_{j=1}$ we have 
\[
  \Bigg(\mathbb{E}\Big\| \sum_{j=1}^n \varepsilon_jd_j \Big\|^p \Bigg)^{\frac{1}{p}} \leq \beta_{p,E}\Bigg(\mathbb{E}\Big\| \sum_{j=1}^n d_j \Big\|^p \Bigg)^{\frac{1}{p}}.
\]

The following lemma gives necessary and sufficient conditions for a process $\Phi$ to be stochastically integrable by characterising integrability in terms of representation of $\gamma$-radonifying operators.

\begin{lemma}[Lemma 2.4 of \cite{VanN-Weis_1}]
 \label{lemma:ness&suff_stock_integrable}
Let $E$ be a UMD space. For an $H$-strongly measurable process $\Phi:(0,T)\times\Omega\to\mathcal{L}(H,E)$, the following are equivalent:
\begin{itemize}
 \item[(i)] the process $\Phi$ is stochastically integrable with respect to $W_H$;
 \item[(ii)] for all $x^*\in E^*$ the process $\Phi^*x^*$ belongs to $L^0\big(\Omega;L^2(0,T;H)\big)$  and there exists an operator valued random variable $R:\Omega \to \gamma\big(L^2(0,T;H),E\big)$ such that for all $f\in L^2(0,T;H)$ and $x^*\in E^*$ we have 
\[
 x^*(Rf) = \int_0^T \langle f(t), \Phi^*(t)x^* \rangle_H \ud t \quad\text{in } L^0(\Omega).
\]
\end{itemize}
\end{lemma}
In this case we say that $R$ is represented by $\Phi$. We will often identify $\Phi$ with $R$ without further comment, for example when referring to the norm $\|\Phi(\cdot,\omega)\|_{\gamma(L^2(0,T;H),E)}$.

\begin{defn}
\label{def:R-bdd}
 Let $E$ and $F$ be Banach spaces and suppose $(r_n)_{n\geq1}$ is a \emph{Rademacher sequence}, that is, a sequence of independent random variables $r_n$  with $\mathbb{P}\{r_n = -1\} = \mathbb{P}\{r_n = 1\} = 1/2$. A family $\mathscr{T}\subset \mathcal{L}(E,F)$ is said to be \emph{$R$-bounded} if there exists a constant $C\geq 0$ such that for all finite sequences $(x_n)_{n=1}^N$ in $E$ and $(T_n)_{n=1}^N$ in $\mathscr{T}$ we have 
\begin{equation}
\label{eqn:def_R-bdd}
 \mathbb{E}\Big\|\sum_{n=1}^N r_n T_n x_n \Big\|^2 \leq C^2 \mathbb{E}\Big\| \sum_{n=1}^N r_n x_n \Big\|^2.
\end{equation}
The least such constant $C$ is called the \emph{$R$-bound} of $\mathscr{T}$ and denoted $R(\mathcal{T})$.

If a family $\mathscr{T}$ satisfies the inequality \eqref{eqn:def_R-bdd} with the Rademacher sequence replaced by a sequence $(\gamma_n)_{n\geq1}$ of standard Gaussian random variables then we say $\mathscr{T}$ is \emph{$\gamma$-bounded}, and denote the $\gamma$-bound by $\gamma(\mathscr{T})$ in the same way. By a standard randomisation argument every $R$-bounded family is also $\gamma$-bounded. The converse holds if the range space if of finite cotype. See \cite{VanN-Weis_1} Section 3.

\end{defn}

The following Lemma is essentially Lemma 2.9 of \cite{VanN-Weis_1}. The restriction to spaces which do not contain $c_0$ does not appear in the reference, but following discussion with the author it is believed to be necessary. It should be noted that no UMD space contains a copy of $c_0$.

\begin{lemma}
\label{lemma:2.9}
Let $E,F$ be Banach spaces which do not contain closed subspaces isomorphic to $c_0$, $H$ be a separable Hilbert space and $T>0$. Let $M:(0,T)\to\mathcal{L}(E,F)$ be a function satisfying
\begin{itemize}
\item[(i)] for all $x\in E$ the function $M(\cdot)x$ is strongly measurable in $F$;
\item[(ii)] the range $\mathscr{M} := \{M(t): t\in(0,T)\}$ is $\gamma$-bounded in $\mathcal{L}(E,F)$.
\end{itemize}
Then for all step functions $\Psi:(0,T)\to\mathcal{L}(H,E)$ taking values in the finite rank operators from $H$ to $E$ we have
\begin{equation}
\label{eqn:lemma2.9}
  \|M\Psi\|_{\gamma(L^2(0,T;H),F)} \leq \gamma(\mathscr{M})\|\Psi\|_{\gamma(L^2(0,T;H),E)}.
\end{equation}
Here $(M\Psi)(t):= M(t)\Psi(t)$. The mapping $\Psi\mapsto M\Psi$ then has a unique extension to a bounded operator from $\gamma\big(L^2(0,T;H),E\big)$ to $\gamma\big(L^2(0,T;H),F\big)$ of norm at most $\gamma(\mathscr{M})$.
\end{lemma}

For the deterministic and stochastic convolutions of the semigroup $S$ with a function $\phi: [0,T]\to E$ we write $S*\phi$ and $S\diamond\phi$ respectively, defined by
\[
 \big(S*\phi\big)(t) := \int_0^t S(t-s)\phi(s) \ud s, \quad \big(S\diamond\phi\big)(t) := \int_0^t S(t-s)\phi(s) \ud W_H(s).
\]
We will refer to several convolution estimates from \cite{VanN-Weis_1}, sections 3 and 4 in the course of our results.

A Banach space $E$ is said to have \emph{type} $p\in[1,2]$ if there exists a constant $C\geq0$ such that for all finite sequences $(x_n)_{n=1}^N$ in $E$ we have
\[
 \Bigg( \mathbb{E}\Bigg\|\sum_{n=1}^N r_n x_n \Bigg\|^2\Bigg)^{\frac{1}{2}} \leq C \Bigg(\sum_{n=1}^N \|x_n\|^p\Bigg)^{\frac{1}{p}}
\]
where $(r_n)_{n\geq1}$ is a Rademacher sequence as in Definition \ref{def:R-bdd}. Every Banach space has type $1$, the spaces $L^p(S)$ have type $\min\{p,2\}$ and every UMD space has type $p>1$. Note in the results below, that the valid ranges of some constants may depend on the type of $E$.

Finally for this section, we recall the notion of $L^2_\gamma$-Lipschitz functions. This assumption is stronger than the standard Lipschitz property and affords a way out of the difficulties alluded to in the introduction.

\begin{defn}
\label{def:gamma-lipshitz}
Let $(S,\Sigma)$ be a countably generated measurable space and let $\mu$ be a finite measure on $(S,\Sigma)$. Then $L^2(S,\mu)$ is separable and we may define 
\[
L^2_\gamma(S,\mu;E) := \gamma\big(L^2(S,\mu),E\big) \cap L^2(S,\mu;E)
\]
as the space of all strongly $\mu$-measurable functions $\phi:S\to E$ for which
\begin{equation}
\label{eqn:L^2_gamma-norm}
\|\phi\|_{L^2_\gamma(S,\mu;E)} := \|\phi\|_{\gamma(L^2(S,\mu),E)} + \|\phi\|_{L^2(S,\mu;E)} 
\end{equation}
is finite. One easily checks that the simple functions $S\to E$ are dense in $L^2_\gamma(S,\mu;E)$.

For non-zero separable Hilbert space $H$ and Banach spaces $E,F$, we say a strongly continuous function $f:S\times E\to\mathcal{L}(H,F)$ is \emph{$L^2_\gamma$-Lipschitz with respect to $\mu$} if for each $x\in E$ we have $f(\cdot,x)\in \gamma\big(L^2(S,\mu;H),F\big)$ and for all simple functions $\phi,\psi:S\to E$ we have 
\begin{equation}
\label{eqn:gamma-lipshitz}
  \|f(\cdot,\phi) - f(\cdot,\psi)\|_{\gamma(L^2(S,\mu;H),F)} \leq C\|\phi - \psi\|_{L^2_\gamma(S,\mu;E)}.
\end{equation}
\end{defn}

\begin{lemma}[Lemma 5.1 of \cite{VanN-Weis_1}]
If $f:S\times E\to\mathcal{L}(H,F)$ is an $L^2_\gamma$-Lipschitz function with respect to $\mu$, then the operator $\phi \mapsto f(\cdot,\phi(\cdot))$ extends uniquely to a Lipshitz mapping from $L^2_\gamma(S,\mu;E)$ into $\gamma\big(L^2(S,\mu;H),F\big)$ represented by the function $f(\cdot,\phi(\cdot))$.
\end{lemma}

For an important class of examples, the \emph{Nemytskii maps} we refer to \cite{VanN-Weis_1} Example 5.5.

\section{Results}
\label{sec:Results}

Consider the following delayed stochastic Cauchy problem
\begin{equation}
\label{eqn:DSDE}
\left\{
\begin{array}{rll}
  \ud U(t) &= &\big(AU(t) + F(t,U_t)\big) \ud t + B(t,U_t)\ud W_H(t),\; t\in[0,T_0]\\
  U_0&= &\Phi.
\end{array} \right.
\end{equation}
We make the following assumptions on $A, F, B, U_0$ and the constants $\eta, \theta_F, \theta_B \geq 0$, $\tau\in(1,2]$ and $p>2$:
\begin{itemize}
\item[(D1)] $A$ generates an analytic $C_0$-semigroup on a UMD Banach space $E$ of type $\tau$ and $W_H = \big(W_H(t)\big)_{t\in[0,T_0]}$ 
is an $H$-cylindrical Brownian motion on Hilbert space $H$.
\item[(D2)] $0\leq \eta + \theta_F < \frac{3}{2} - \frac{1}{\tau}$ and $0<\eta + \theta_B + \frac{1}{p} < \frac{1}{2}$.
\item[(D3)] The history space $\mathscr{B}$ is of the form $L^p_g(-\infty,0;E)\times E$ introduced in Example \ref{ex:weightedhistoryspace}
\item[(D4)] The function $F:[0,T_0]\times\Omega\times\mathscr{B}_\eta \to E_{-\theta_F}$ is continuous on $[0,T_0]$ and is Lipschitz and of linear growth on the history space $\mathscr{B}$ uniformly with respect to $[0,T_0]\times\Omega$ . In other words, there exist constants $L_F, C_F \geq 0$ such that for all $t\in[0,T_0], \omega\in \Omega$ and $\phi, \psi \in \mathscr{B}_\eta$ we have 
  \begin{equation}
    \label{eqn:F-lipshitz}
    \|F(t,\omega,\phi) - F(t,\omega,\psi)\|_{E_{-\theta_F}} \leq L_F\|\phi - \psi\|_{\mathscr{B}_\eta}
  \end{equation}
and
  \begin{equation}
    \label{eqn:F-linear-growth}
    \|F(t,\omega,\phi)\|_{E_{-\theta_F}} \leq C_F\big(1+\|\phi\|_{\mathscr{B}_\eta}\big)
  \end{equation}
and moreover, that for all $\phi\in\mathscr{B}_\eta$ the function $(t,\omega)\mapsto F(t,\omega,\phi)$ is measurable and adapted (to $\mathscr{F}_t$) in $E_{-\theta_F}$.
\item[(D5)] The function $B:[0,T_0]\times\Omega\times\mathscr{B}_\eta \to \mathcal{L}(H,E_{-\theta_B})$ is continuous on $[0,T_0]$ and $L_\gamma^2$-Lipschitz uniformly on $\Omega$. In other words, there exist constants $L_B^\gamma, C_B^\gamma \geq 0$ such that for any finite measure $\mu$ on $([0,T_0],\mathbb{B}_{[0,T_0]})$, all $\omega\in\Omega$ and all $\phi, \psi :(-\infty,T_0]\to E_\eta$, continuous on $[0,T_0]$ with the property that the functions $(t\mapsto\phi_t)$ and $(t\mapsto\psi_t)$ lie in $L^2_\gamma(0,T_0,\mu;\mathscr{B}_\eta)$ we have
  \begin{align}
    \label{eqn:B-lipshitz}
    \|B(\cdot,\omega,\phi_0) - B(\cdot,\omega,\psi_0)\|&_{\gamma(L^2(0,T_0,\mu;H),E_{-\theta_B})}\\
&\leq L_B^\gamma\|t\mapsto(\phi - \psi)_t\|_{L_\gamma^2(0,T_0,\mu;\mathscr{B}_\eta)}\nonumber
  \end{align}
and
  \begin{equation}
    \label{eqn:B-linear-growth}
    \|B(\cdot,\omega,\phi)\|_{\gamma(L^2(0,T_0,\mu;H),E_{-\theta_B})} \leq C_B^\gamma\big(1+\|t\mapsto\phi_t\|_{L_\gamma^2(0,T_0,\mu;\mathscr{B}_\eta)}\big)
  \end{equation}
and moreover, for all $\phi\in\mathscr{B}_\eta$ the map $(t,\omega)\mapsto B(t,\omega,\phi)$ is $H$-strongly measurable and adapted (to $\mathscr{F}_t$) in $\mathcal{L}(H,E_{-\theta_B})$.
\item[(D6)] The initial history data $U_0 = \Phi:\Omega  \to \curlyB_\eta$ is strongly $\mathscr{F}_0$-measurable, $\Phi(0)\in L^p(\Omega;E_\eta)$ and if $\widehat{\Phi}$ represents the extension of $\Phi$ to $(-\infty,T_0]$ by
\begin{equation}
\label{eqn:defHat}
 \widehat{\Phi}(t) := \left\{ \begin{array}{ll}
			  0 & t\in (0,T_0]\\
			  \Phi(t) & t\in(-\infty,0].
			\end{array}\right.
\end{equation}
then the function mapping $s$ to $(t-s)^{-\alpha}\widehat{\Phi}(\omega)_s$ lies in $L^2_\gamma\big(0,t;\curlyB_\eta\big)$ for all $t\in[0,T_0]$, $\alpha\in(0,1/2)$ and almost all $\omega\in\Omega$ with
\begin{equation}
\label{eqn:Phi_gamma_lipschitz}
 \sup_{t\in[0,T_0]}\mathbb{E}\big\|s\mapsto(t-s)^{-\alpha}\widehat{\Phi}_s\big\|_{L^2_\gamma(0,t;\curlyB_\eta)} <\infty.
\end{equation}
\end{itemize}

\begin{defn}
A process $\big(U(t)\big)_{t\in(-\infty,T_0]}$ is said to be a \emph{mild solution} of \eqref{eqn:DSDE} if $U: [0,T_0]\times\Omega\to E_\eta$ is strongly measurable, adapted to $\mathscr{F}_t$ and for all $t\in[0,T_0]$
\begin{itemize}
\item[(i)] the function $s\mapsto S(t-s)F(s,U_s)$ is in $L^0\big(\Omega; L^1(0,t;E)\big)$;
\item[(ii)] the function $s\mapsto S(t-s)B(s,U_s)$ is $H$-strongly measurable, adapted to $\mathscr{F}_t$ and in $\gamma\big(L^2(0,t;H),E\big)$ almost surely;
\item[(iii)] $U(t) = S(t)\Phi(0) + S*F(\cdot,U_\cdot)(t) + S \diamond B(\cdot,U_\cdot)(t)$ almost surely.
\end{itemize}
\end{defn}

Hence for a given solution $U$ the deterministic convolution is pathwise well defined as a Bochner integral and the stochastic convolution is well defined by \eqref{lemma:ness&suff_stock_integrable}.

We will prove an existence and uniqueness result for \eqref{eqn:DSDE} using a fixed point argument in a scale of Banach spaces introduced by van Neerven et al. in \cite{VanN-Weis_1}. Fix $T\in(0,T_0]$, $p\in[1,\infty)$ and $\alpha\in(0,1/2)$ and define $V^p_{\alpha,\infty}\big([0,T]\times\Omega;E\big)$ to be the space of all continuous adapted processes $\phi:[0,T]\times\Omega \to E$ for which 
\begin{align}
\label{eqn:defnVnorm}
\|\phi&\|_{V^p_{\alpha,\infty}([0,T]\times\Omega;E)}\\
  &:= \Big(\mathbb{E}\|\phi\|^p_{C([0,T];E)}\Big)^{\frac{1}{p}} + \sup_{t\in[0,T]}\Big(\mathbb{E}\big\|s\mapsto (t-s)^{-\alpha}\phi(s)\big\|^p_{\gamma(L^2(0,t),E)}\Big)^{\frac{1}{p}}\nonumber
\end{align}
is finite. Similarly, define $V^p_{\alpha,p}\big([0,T]\times\Omega;E\big)$ to be the space of pathwise continuous and adapted processes $\phi:[0,T]\times\Omega \to E$ for which 
\begin{align} 
\label{eqn:defnV_pnorm}
&\|\phi\|_{V^p_{\alpha,p}([0,T]\times\Omega;E)}\\
  &:= \Big(\mathbb{E}\|\phi\|^p_{C([0,T];E)}\Big)^{\frac{1}{p}} + \Bigg(\int_0^T\mathbb{E}\big\|s\mapsto (t-s)^{-\alpha}\phi(s)\big\|^p_{\gamma(L^2(0,t),E)}\ud t\Bigg)^{\frac{1}{p}} < \infty.\nonumber
\end{align}
Under identification of processes which are indistinguishable under the above norms, $V^p_{\alpha,\infty}\big([0,T]\times\Omega;E\big)$ and $V^p_{\alpha,p}\big([0,T]\times\Omega;E\big)$ become Banach spaces. 

Our main result, Theorem \ref{thm:6.2}, is to prove the existence and uniqueness of a mild solution of \eqref{eqn:DSDE} in each of the spaces above with the initial history data $\Phi$.

Consider the fixed point operator defined on $V^p_{\alpha,\infty}$ or $V^p_{\alpha,p}$ by
\begin{equation}
\label{eqn:fixedpointoperator}
L_T(\phi)(t) := S(t)\Phi(0) + S*F(\cdot,\widetilde{\phi}_\cdot)(t) + S \diamond B(\cdot,\widetilde{\phi}_\cdot)(t),\quad t\in[0,T]
\end{equation}
where $\widetilde{\phi}$ is the extension of $\phi$ to $(-\infty,T]$ by $\Phi$,
\begin{equation}
\label{eqn:deftilde}
\widetilde{\phi}(t) := \left\{ \begin{array}{ll}
			  \phi(t) & t\in [0,T]\\
			  \Phi(t) & t\in(-\infty,0).
			\end{array}\right.
\end{equation}
We will show that $L_T$ is well defined and becomes a strict contraction for small enough $T$.

\begin{prop}
\label{prop:6.1}
Suppose (D1) -- (D6) are satisfied and choose $\alpha\in(0,1/2)$ such that 
\[
\eta + \theta_B < \alpha - \frac{1}{p}.
\]
The operator $L_T$ is well defined and bounded on the spaces $V = V^p_{\alpha,\infty}\big([0,T]\times\Omega;E_\eta\big)$ and $V^p_{\alpha,p}\big([0,T]\times\Omega;E_\eta\big)$, and moreover, there exist constants $C, C_T$ with $\lim_{T\downarrow0}C_T = 0$ such that for all $\phi,\psi\in V$
\begin{equation}
\label{eqn:L_Lipschitz}
\|L_T(\phi) - L_T(\psi)\|_V \leq C_T\|\phi - \psi\|_V
\end{equation}
and 
\begin{align}
\label{eqn:L_linear_growth}
\|L_T(\phi)\|_V &\leq C\Big(1 + \big(\mathbb{E}\|\Phi\|^p_{\curlyB_\eta}\big)^\frac{1}{p} \nonumber\\
&+ \big(\!\!\sup_{t\in[0,T]}\mathbb{E}\big\|s\mapsto(t-s)^{-\alpha}\widehat{\Phi}_s\big\|^p_{L^2_\gamma(0,t;\curlyB)}\big)^{\frac{1}{p}}\Big) + C_T\|\phi\|_V
\end{align}
\end{prop}

First we give two more technical lemmas that will be required in the proof. These are exactly as in the proof of Proposition 6.1 in \cite{VanN-Weis_1}, but are presented separately here for clarity, as the delay plays no part in these estimates.

\begin{lemma}
\label{lemma:deterministic_part_a}
Assume (D1)-(D6). Under the conditions of Proposition \ref{prop:6.1}, if $\Psi\in L^p\big(\Omega;C([0,T];E_{-\theta_F})\big)$ then
\begin{equation}
\label{eqn:deterministic_estimate_3}
  \bn S*\Psi\bn_{V^p_{\alpha,\infty}([0,T]\times\Omega;E_\eta)} \leq CT^{\min\{\frac{1}{2} - \alpha, 1-\eta-\theta_F\}}\|\Psi\|_{L^p(\Omega;C([0,T];E_{-\theta_F}))}.
\end{equation}
\end{lemma}

\begin{proof}
 Take $\psi\in C([0,T];E_{-\theta_F})$ and find bounds for the $V^p_{\alpha,\infty}\big([0,T]\times\Omega;E_\eta\big)$-norm of $S*\psi$. Applying Lemma 3.6 of \cite{VanN-Weis_1} with $\alpha = 1$ and $\lambda = 0$ we see that $S*\psi$ is continuous in $E_\eta$. The standard analytic property  \eqref{eqn:analyticproperty} of $S$ gives 
\begin{align}
\label{eqn:deterministic_estimate_1}
  \bn S*\psi\bn_{C([0,T];E_\eta)} &\leq C\int_0^t (t-s)^{-(\eta+\theta_F)}\ud s \|\psi\|_{C([0,T];E_{-\theta_F})}\nonumber\\
&\leq CT^{1 - \eta - \theta_F}\|\psi\|_{C([0,T];E_{-\theta_F})}.
\end{align}
By assumption (D2) we satisfy the conditions of Proposition 3.5 of \cite{VanN-Weis_1}, so it follows that
\begin{equation}
\label{eqn:deterministic_estimate_2}
\bn s\mapsto (t-s)^{-\alpha} S*\psi(s) \bn_{\gL} \leq T^{\frac{1}{2}-\alpha}\|\psi\|_{C([0,T];E_{-\theta_F})}.
\end{equation}
Now let $\Psi\in L^p\big(\Omega;C([0,T];E_{-\theta_F})\big)$. We apply \eqref{eqn:deterministic_estimate_1} and \eqref{eqn:deterministic_estimate_2} to the paths $\Psi(\cdot,\omega)$ and take expectations to see that $S*\Psi\in V^p_{\alpha,\infty}\big([0,T]\times\Omega;E\big)$ and
\[
  \bn S*\Psi\bn_{V^p_{\alpha,\infty}([0,T]\times\Omega;E_\eta)} \leq CT^{\min\{\frac{1}{2} - \alpha, 1-\eta-\theta_F\}}\|\Psi\|_{L^p(\Omega;C([0,T];E_{-\theta_F}))}.\qedhere
\]
\end{proof}

\begin{lemma}
\label{lemma:stochastic_part_a}
Assume (D1) - (D6). Under the conditions of Proposition \ref{prop:6.1} let $\Psi:[0,T]\times\Omega\to \mathcal{L}(H,E_{-\theta_B})$ be $H$-strongly measurable and adapted, and suppose that 
\begin{equation}
\label{eqn:Psi_assumption1}
\sup_{t\in[0,T]}\mathbb{E}\big\|s\mapsto(t-s)^{-\alpha}\Psi(s)\big\|^p_{\gamma(L^2(0,t;H),E_{-\theta_B})} < \infty.
\end{equation}
Then there exists $\varepsilon' >0$ such that
\begin{align}
\label{eqn:stochastic_estimate1}
\big\|S&\diamond\Psi\big\|_{V^p_{\alpha,\infty}([0,T]\times\Omega;E_\eta)} \\
& \leq \quad CT^{\varepsilon'} \Big(\sup_{t\in[0,T]}\mathbb{E}\big\|s\mapsto(t-s)^{-\alpha}\Psi(s)\big\|^p_{\gamma(L^2(0,t;H),E_{-\theta_B})}\Big)^{\frac{1}{p}}.\nonumber
\end{align}
\end{lemma}

\begin{proof}
 Let $\Psi:[0,T]\times\Omega\to \mathcal{L}(H,E_{-\theta_B})$ be $H$-strongly measurable and adapted, and suppose that 
\begin{equation}
\sup_{t\in[0,T]}\mathbb{E}\big\|s\mapsto(t-s)^{-\alpha}\Psi(s)\big\|^p_{\gamma(L^2(0,t;H),E_{-\theta_B})} < \infty.
\end{equation}
We estimate the $V^p_{\alpha,\infty}\big([0,T]\times\Omega;E_\eta\big)$-norm of $S\diamond\Psi$. From Proposition 4.2 of \cite{VanN-Weis_1} there exists an $\varepsilon>0$ such that 
\[
\mathbb{E}\big\|S\diamond\Psi\big\|^p_{C([0,T];E_\eta)} \leq C^pT^{\varepsilon p} \sup_{t\in[0,T]}\mathbb{E}\big\|s\mapsto(t-s)^{-\alpha}\Psi(s)\big\|^p_{\gamma(L^2(0,t;H),E_{-\theta_B})}.
\]
For the other part of the norm, by Proposition 4.5 of \cite{VanN-Weis_1} we obtain that
\begin{align}
\mathbb{E}\big\|s\mapsto&(t-s)^{-\alpha}S\diamond\Psi(s)\big\|^p_{\gamma(L^2(0,t;H),E_{\eta})}\nonumber\\
& \leq \quad C^pT^{(\frac{1}{2}-\eta-\theta_B)p} \mathbb{E} \big\|s\mapsto(t-s)^{-\alpha}\Psi(s)\big\|^p_{\gamma(L^2(0,t;H),E_{-\theta_B})}.\nonumber
\end{align}
Combining the above we conclude that for $\varepsilon' := \min\{1/2-\eta-\theta_B,\varepsilon\}$
\begin{align}
\big\|S&\diamond\Psi\big\|_{V^p_{\alpha,\infty}([0,T]\times\Omega;E_\eta)} \\
& \leq \quad CT^{\varepsilon'} \Big(\sup_{t\in[0,T]}\mathbb{E}\big\|s\mapsto(t-s)^{-\alpha}\Psi(s)\big\|^p_{\gamma(L^2(0,t;H),E_{-\theta_B})}\Big)^{\frac{1}{p}}.\nonumber\qedhere
\end{align}
\end{proof}

We are now ready to proceed with the main result.

\begin{proof}[Proof (of Proposition \ref{prop:6.1})]
We proceed along the same lines as the proof of Proposition 6.1 in \cite{VanN-Weis_1}, adapting as necessary for the delay. As in \cite{VanN-Weis_1} we will give a detailed proof only in the case that $V = V^p_{\alpha,\infty}\big([0,T]\times\Omega;E\big)$, the other case being essentially the same.

\mbox{}\\
\emph{Step 1} (Estimating the initial part).
It is clear that for $\omega\in\Omega$
\[
\big\|S\Phi(0)\bn_{C([0,T];E_\eta)} \leq C\|\Phi(0)\|_{E_\eta}.
\]
Let $\varepsilon \in (0,1/2).$ From Lemma \ref{lemma:2.9} and Proposition 4.1 of \cite{VanN-Weis_1} we infer that for a fixed $\omega\in\Omega$, $t\in[0,T]$
\begin{align}
\big\|s \mapsto (t-s)^{-\alpha}S(s)\Phi(0)\big\|&_{\gL} \nonumber\\
&\leq \quad C\big\|s \mapsto (t-s)^{-\alpha} s^{-\varepsilon}\Phi(0)\big\|_{\gL}\nonumber\\
&\leq \quad C\big\|s \mapsto (t-s)^{-\alpha}s^{-\varepsilon}\big\|_{L^2(0,t)}\|\Phi(0)\|_{E_\eta}\nonumber\\
&\leq \quad C \|\Phi(0)\|_{E_\eta}.\nonumber
\end{align}
Then taking expectations, we get 
\begin{equation}
\label{eqn:initial_estimate}
\bn S\Phi(0)\bn_{V^p_{\alpha,\infty}([0,T]\times\Omega;E)} \leq C\|\Phi(0)\|_{L^p(\Omega;E_\eta)}.
\end{equation}
\mbox{}\\
\emph{Step 2} (Estimating the deterministic convolution).

Let $\phi,\psi\in V^p_{\alpha,\infty}\big([0,T]\times\Omega;E_\eta\big)$. Since $F$ is continuous on $[0,T]\times\curlyB_\eta$ and the mapping $t\mapsto\phi_t$ is continuous by Lemma \ref{lemma:def:phasespace}, both $F(\cdot,\widetilde{\phi}_\cdot)$ and $F(\cdot,\widetilde{\psi}_\cdot)$
belong to $L^p\big(\Omega;C([0,T];E_{-\theta_F})\big)$ and we can apply Lemma \ref{lemma:deterministic_part_a}. The estimate \eqref{eqn:deterministic_estimate_3} shows that $S*F(\cdot,\widetilde{\phi}_\cdot)$ and $S*F(\cdot,\widetilde{\psi}_\cdot) \in V^p_{\alpha,\infty}\big([0,T]\times\Omega;E_\eta\big)$. Note that by the property \eqref{eqn:KM_ineq} of Lemma \ref{lemma:def:phasespace}
\begin{equation}
\label{eqn:curlyB_estimate}
\|\widetilde{\phi}_t-\widetilde{\psi}_t\|_{\curlyB_\eta} \leq \sup_{s\in[0,T]} K(s)\|\phi-\psi\|_{C([0,T],E_{\eta})}
\end{equation}
since $\widetilde{\phi}_0 = \widetilde{\psi}_0 = \Phi$. Combining \eqref{eqn:deterministic_estimate_3} with the property that $F$ is Lipschitz in its $\curlyB_\eta$-variable \eqref{eqn:F-lipshitz} and setting $\delta = \min\{1/2 - \alpha, 1-\eta-\theta_F\}$ yields
\begin{align}
\label{eqn:F_contraction}
\bn S*\big(F(\cdot,\widetilde{\phi}_\cdot) - &F(\cdot,\widetilde{\psi}_\cdot)\big)\bn_{V^p_{\alpha,\infty}([0,T]\times\Omega;E_\eta)}\nonumber\\
& \leq \quad CT^\delta \bn F(\cdot,\widetilde{\phi}_\cdot) - F(\cdot,\widetilde{\psi}_\cdot)\bn_{L^p(\Omega;C([0,T];E_{-\theta_F}))} \nonumber\\
& = \quad CT^\delta \Big(\mathbb{E}\sup_{t\in[0,T]} \bn F(t,\widetilde{\phi}_t) - F(t,\widetilde{\psi}_t)\bn_{E_{-\theta_F}}^p \Big)^{\frac{1}{p}} \nonumber\\
& \leq \quad CT^\delta L_F \Big(\mathbb{E}\sup_{t\in[0,T]} \|(\widetilde{\phi}-\widetilde{\psi})_t\|_{\curlyB_\eta}^p\Big)^{\frac{1}{p}}\nonumber\\
& \leq \quad CT^\delta L_F K_{T_0}\Big(\mathbb{E} \|\phi - \psi\|_{C([0,T];E_\eta)}^p\Big)^{\frac{1}{p}}\nonumber\\
& \leq \quad CT^\delta L_F K_{T_0}\|\phi - \psi\|_{V^p_{\alpha,\infty}([0,T]\times\Omega;E_\eta)}.
\end{align}
\mbox{}\\
\emph{Step 3} (Estimating the stochastic convolution).

For $t\in[0,T]$, let $\mu_{t,\alpha}$ be the finite measure on $\big((0,t),\mathbb{B}_{(0,t)}\big)$ defined by
\[
\mu_{t,\alpha}(B) := \int_0^t (t-s)^{-2\alpha}\mathbbm{1}_B(s) \ud s.
\]
Notice that for a function $\phi\in C\big([0,t];E\big)$ we have 
\[
\phi\in\gamma\big(L^2(0,t,\mu_{t,\alpha}),E\big) \quad\Longleftrightarrow\quad \big[s\mapsto(t-s)^{-\alpha}\phi(s)\big]\in\gamma\big(L^2(0,t),E\big).
\]
and
\begin{align}
\|\phi\|_{L^2(0,t,\mu_{t,\alpha};E)} \quad &= \quad \big\|s\mapsto(t-s)^{-\alpha}\phi(s)\big\|_{L^2(0,t;E)} \nonumber\\
& \leq \quad Ct^{\frac{1}{2} - \alpha}\|\phi\|_{C([0,T];E)}.\nonumber
\end{align}

Now let $\phi,\psi \in V^p_{\alpha,\infty}\big([0,T]\times\Omega;E_\eta\big)$. Fix some $\omega\in\Omega$. The paths of $\phi, \psi$ belong to $L^2_\gamma\big(0,t,\mu_{t,\alpha};E_\eta\big)$ uniformly for $t\in[0,T]$ almost surely by the definition of the norm on $V^p_{\alpha,\infty}\big([0,T]\times\Omega;E_\eta\big)$ \eqref{eqn:defnVnorm} and that the path of $s\mapsto\widehat{\Phi}_s$ is in $L^2_\gamma\big(0,t,\mu_{t,\alpha};\curlyB_\eta\big)$ almost surely by assumption (D6). 

Recall the definition of $\widetilde{\phi}, \widetilde{\psi}$ from \eqref{eqn:deftilde}. We plan to show that $s\mapsto\widetilde{\phi}_s$ and $s\mapsto\widetilde{\psi}_s$ are contained in $L^2_\gamma\big(0,t,\mu_{t,\alpha};\curlyB_\eta\big)$ almost surely. From the definition \eqref{eqn:L^2_gamma-norm} we have that
\begin{align}
\big\|s\mapsto\widetilde{\phi}_s\big\|_{L^2_\gamma(0,t,\mu_{t,\alpha};\curlyB_\eta)} = \big\|s\mapsto\widetilde{\phi}_s\big\|_{L^2(0,t,\mu_{t,\alpha};\curlyB_\eta)}
+ \big\|s\mapsto \widetilde{\phi}_s\big\|_{\gamma(L^2(0,t,\mu_{t,\alpha}),\curlyB_\eta)}.\nonumber
\end{align}
For the first part, using property \eqref{eqn:KM_ineq} of Lemma \ref{lemma:def:phasespace}
\begin{align}
\label{eqn:tilde_gamma_lipshitz_1}
\big\|s\mapsto\widetilde{\phi}_s\big\|^2_{L^2(0,t,\mu_{t,\alpha};\curlyB_\eta)} \quad&=\quad
  \int_0^t\|\widetilde{\phi}_s\|^2_{\curlyB_\eta} \ud \mu_{t,\alpha}(s)\nonumber\\
&\leq\quad \int_0^t K^2_T\|\phi\|^2_{C([0,T];E_\eta)} + M^2_T\|\widetilde{\phi}_0\|^2_{\curlyB_\eta} \ud \mu_{t,\alpha}(s) \nonumber\\
&\leq\quad C \Big(\|\phi\|^2_{C([0,T];E_\eta)} + \|\Phi\|^2_{\curlyB_\eta}\Big)
\end{align}
where $K_T = \sup_{s\in[0,T]}K(s)$, $M_T=\sup_{s\in[0,T]}M(s)$. 
By the corollary to the $\gamma$-Fubini isomorphism (Corollary \ref{cor:product_gamma_fub}), since $\curlyB_\eta = L^p_g\big((-\infty,0];E_\eta\big)\times E_\eta$ for some $g$ we have that 
\begin{align}
\label{eqn:gamma_curlyB_isomorphism}
\gamma\Big(L^2\big(0,t,\mu_{t,\alpha}\big),\curlyB_\eta\Big) &= \gamma\Big(L^2\big(0,t,\mu_{t,\alpha}\big),L^p\big(-\infty,0,g(r) \ud r;E_\eta\big)\times E_\eta\Big) \nonumber\\
&\simeq L^p_g\Big((-\infty,0];\gamma\big(L^2(0,t,\mu_{t,\alpha}),E_\eta\big)\Big)\nonumber\\
&\qquad\qquad\qquad\qquad\qquad\times\gamma\big(L^2(0,t,\mu_{t,\alpha}),E_\eta\big).
\end{align}
We now apply this isomorphism to the map $s\mapsto \widetilde{\phi}_s$
\begin{align}
\label{eqn:tilde_gamma_lipshitz_2}
\big\|s\mapsto&\widetilde{\phi}_s\big\|^p_{\gamma(L^2(0,t,\mu_{t,\alpha}),\curlyB_\eta)}\nonumber\\  
  &\leq C \big\|s\mapsto\widetilde{\phi}_s\big\|^p_{L^p_g((-\infty,0];\gamma(L^2(0,t,\mu_{t,\alpha}),E_\eta))\times\gamma(L^2(0,t,\mu_{t,\alpha}),E_\eta)}\nonumber\\
  &= C\Big( \big\|s\mapsto\widetilde{\phi}_s\big\|^p_{L^p_g((-\infty,0];\gamma(L^2(0,t,\mu_{t,\alpha}),E_\eta))} + \big\|s\mapsto\widetilde{\phi}_s(0)\big\|^p_{\gamma(L^2(0,t,\mu_{t,\alpha}),E_\eta)}\Big)\nonumber\\
&= C\big( I_1 + I_2\big).\nonumber
\end{align}
To deal with the $L^p_g$ term $I_1$ we partition the domain $(0,t)$ of $\widetilde{\phi}_s$ so as to approximate the initial history part $\Phi$ and the part in $\V$ separately.
\begin{align}
 I_1 &= C \int_{-\infty}^0 \!g(r)\big\|s\mapsto(t-s)^{-\alpha}\widetilde{\phi}_s(r)\big\|^p_{\gL} \ud r \nonumber\\
  &= C \int_{-\infty}^0 \!g(r)\big\|s\mapsto(t-s)^{-\alpha}\widetilde{\phi}(s+r)\big\|^p_{\gL} \ud r \nonumber\\
  &= C \int_{-\infty}^0 \!g(r)\big\|u\mapsto(t+r-u)^{-\alpha}\widetilde{\phi}(u)\big\|^p_{\gamma(L^2(r,t+r),E_\eta)} \ud r \nonumber\\
  &= C \int_{-\infty}^0 g(r)\Big\|u\mapsto(t+r-u)^{-\alpha}\Big[\mathbbm{1}_{[r,0\wedge(t+r)]}(u)\Phi(u) \nonumber\\
  & \qquad\qquad + \mathbbm{1}_{(0\wedge(t+r),t+r]}(u)\phi(u)\Big]\Big\|^p_{\gamma(L^2(r,t+r),E_\eta)} \ud r \nonumber\\ 
  &\leq C \int_{-\infty}^0 \!g(r)\Bigg[\big\|u\mapsto(t+r-u)^{-\alpha}\Phi(u)\big\|^p_{\gamma(L^2(r,0\wedge(t+r)),E_\eta)}  \nonumber\\
    & \qquad\qquad + \big\|u\mapsto(t+r-u)^{-\alpha}\phi(u)\big\|^p_{\gamma(L^2(0\wedge(t+r),t+r),E_\eta)}\Bigg] \ud r.\nonumber
\end{align}
Now the second term in this integral is $0$ for $r<-t$, so
\begin{align}
  &\leq C \Bigg[\int_{-\infty}^0 \!g(r)\big\|u\mapsto(t+r-u)^{-\alpha}\widehat{\Phi}(u)\big\|^p_{\gamma(L^2(r,t+r),E_\eta)} \ud r  \nonumber\\ 
  & \qquad\qquad + \int_{-t}^0 \!g(r)\big\|u\mapsto(t+r-u)^{-\alpha}\phi(u)\big\|^p_{\gamma(L^2(0,t+r),E_\eta)} \ud r \Bigg]\nonumber\\  
  &\leq C \Bigg[\int_{-\infty}^0\!g(r)\big\|s\mapsto (t-s)^{-\alpha}\widehat{\Phi}_s(r)\big\|^p_{\gamma(L^2(0,t),E_\eta)}\ud r\nonumber\\
  & \qquad\qquad + \sup_{\substack{t\in[0,T]\\r\in[-t,0]}}\big\|u\mapsto(t+r-u)^{-\alpha}\phi(u)\big\|^p_{\gamma(L^2(0,t+r),E_\eta)} \int_{-t}^0 \!g(r)\ud r \Bigg]\nonumber\\  
  &\lesssim C\Bigg[ \sup_{t\in[0,T]}\big\|s\mapsto (t-s)^{-\alpha}\widehat{\Phi}_s\|^p_{\gamma(L^2(0,t),\curlyB_\eta)}\nonumber\\
  & \qquad\qquad + \sup_{t\in[0,T]}\big\|s\mapsto(t-s)^{-\alpha}\phi(s)\big\|^p_{\gamma(L^2(0,t),E_\eta)} \int_{-t}^0 \!g(r)\ud r \Bigg], 
\end{align}
and since $\phi\in\V$ and $\Phi$ satisfies \eqref{eqn:Phi_gamma_lipschitz}, this is integrable with finite expectation. For $I_2$
\begin{align}
 I_2 &= \big\|s\mapsto\widetilde{\phi}_s(0)\big\|^p_{\gamma(L^2(0,t,\mu_{t,\alpha}),E_\eta)} = \big\|s\mapsto\phi(s)\big\|^p_{\gamma(L^2(0,t,\mu_{t,\alpha}),E_\eta)}\nonumber\\
&\leq  \sup_{t\in[0,T]}\big\|s\mapsto(t-s)^{-\alpha}\phi(s)\big\|^p_{\gamma(L^2(0,t),E_\eta)}
\end{align}
which has finite expectation for the same reason. Hence  $s\mapsto\widetilde{\phi}_s$ and $s\mapsto\widetilde{\psi}_s$ are contained in $L^2_\gamma\big(0,t,\mu_{t,\alpha};\curlyB_\eta\big)$ almost surely.

By (D5), $\widetilde{\phi}\mapsto B(\cdot,\widetilde{\phi}_\cdot)$ is of linear growth from the space $L^2_\gamma\big(0,T,\mu_{t,\alpha};\curlyB_\eta\big)$ into $\gamma(L^2(0,T,\mu_{t,\alpha};H),E_{\theta_B}\big)$, and so $\Psi(s) := B(s,\widetilde{\phi}_s)$ satisfies \eqref{eqn:Psi_assumption1} and we can apply Lemma \ref{lemma:stochastic_part_a}. Thus as $B(s,\widetilde{\phi}_s)$ and $B(s,\widetilde{\psi}_s)$ are $H$-strongly measurable and adapted and $S\diamond B(\cdot,\widetilde{\phi}_\cdot)$ and $S\diamond B(\cdot,\widetilde{\psi}_\cdot)$ are in $\V$ by \eqref{eqn:stochastic_estimate1}.

Now using the $L^2_\gamma$-Lipschitz property of $B$ \eqref{eqn:B-lipshitz} 
\begin{align}
 \label{eqn:B_contraction1}
  \big\|&S\diamond\big(B(\cdot,\widetilde{\phi}_\cdot) - B(\cdot,\widetilde{\psi}_\cdot)\big)\big\|_{\subV} \nonumber\\
   &\quad \stackrel{\eqref{eqn:stochastic_estimate1}}{\lesssim} T^{\varepsilon'}\sup_{t\in[0,T]}\Big( \mathbb{E}\big\|s\mapsto(t-s)^{-\alpha}\big[B(s,\widetilde{\phi}_s) - B(s,\widetilde{\psi}_s)\big]\big\|^p_{\gamma(L^2(0,t;H);E_{-\theta_B})}\Big)^{\frac{1}{p}}\nonumber\\
   &\quad \;\,=\;\, T^{\varepsilon'}\sup_{t\in[0,T]}\Big( \mathbb{E}\big\|s\mapsto B(s,\widetilde{\phi}_s) - B(s,\widetilde{\psi}_s)\big\|^p_{\gamma(L^2(0,t,\mu_{t,\alpha};H);E_{-\theta_B})}\Big)^{\frac{1}{p}}\nonumber\\
   &\quad \stackrel{\eqref{eqn:B-lipshitz}}{\lesssim} L^\gamma_B T^{\varepsilon'}\sup_{t\in[0,T]}\Big( \mathbb{E}\big\|s\mapsto (\widetilde{\phi} - \widetilde{\psi})_s\big\|^p_{L^2_\gamma(0,t,\mu_{t,\alpha};\curlyB_\eta)}\Big)^{\frac{1}{p}}\nonumber\\
   &\quad \;\lesssim\; L^\gamma_B T^{\varepsilon'}\Bigg[\sup_{t\in[0,T]}\Big( \mathbb{E}\big\|s\mapsto (t-s)^{-\alpha} (\widetilde{\phi} - \widetilde{\psi})_s\big\|^p_{\gamma(L^2(0,t),\curlyB_\eta)}\Big)^{\frac{1}{p}} \\
&\quad\qquad\qquad\qquad\qquad\qquad + \sup_{t\in[0,T]}\Big( \mathbb{E}\big\|s\mapsto (t-s)^{-\alpha} (\widetilde{\phi} - \widetilde{\psi})_s\big\|^p_{L^2(0,t;\curlyB_\eta)}\Big)^{\frac{1}{p}}\Bigg].\nonumber
\end{align}
Using property \eqref{eqn:KM_ineq} of Lemma \ref{lemma:def:phasespace} and the fact that $\widetilde{\phi}_0 = \widetilde{\psi}_0 = \Phi$ almost surely, we have that for almost all $\omega\in\Omega$
\begin{align}
\label{eqn:B_contraction2}
 \big\|s\mapsto (t-s)^{-\alpha} (\widetilde{\phi} - \widetilde{\psi})_s\big\|^p_{L^2(0,t;\curlyB_\eta)}  &\leq T^{(\frac{p}{2} - \alpha p)} \big\|t\mapsto(\widetilde{\phi} - \widetilde{\psi})_t\big\|^p_{C([0,T];\curlyB_\eta)}\nonumber\\
  &\leq K_T^p T^{(\frac{p}{2} - \alpha p)} \big\|\phi-\psi\big\|^p_{C([0,T];E_\eta)}
\end{align}
and by a further use of Corollary \ref{cor:product_gamma_fub}, the isomorphism 
\begin{align}
 \gamma\Big(L^2\big(0,t,\mu_{t,\alpha}\big),\curlyB_\eta\Big) \simeq L^p_g\Big((-\infty,0];\gamma\big(L^2(&0,t,\mu_{t,\alpha}),E_\eta\big)\Big)\nonumber\\
  &\times\gamma\big(L^2(0,t,\mu_{t,\alpha}),E_\eta\big)
\end{align}
gives us
\begin{align}
\label{eqn:B_contraction3}
 \sup_{t\in[0,T]}&\big\|s\mapsto (t-s)^{-\alpha} (\widetilde{\phi} - \widetilde{\psi})_s\big\|^p_{\gamma(L^2(0,t),\curlyB_\eta)} \nonumber\\
  & \lesssim \sup_{t\in[0,T]} \int^0_{-\infty} g(r)\big\|s\mapsto (t-s)^{-\alpha} (\widetilde{\phi} - \widetilde{\psi})(s+r)\big\|^p_{\gamma(L^2(0,t),E_\eta)}\ud r\nonumber\\
  &\qquad\qquad\qquad\qquad + \sup_{t\in[0,T]}\big\|s\mapsto (t-s)^{-\alpha} (\phi - \psi)(s)\big\|^p_{\gamma(L^2(0,t),E_\eta)}\nonumber\\
  & \leq  \big(1+\|g\|_{L^1(-T,0)}\big)\sup_{\substack{t\in[0,T] \\
 r\in[-t,0]}}\big\|s\mapsto (t-s)^{-\alpha} (\widetilde{\phi} - \widetilde{\psi})(s+r)\big\|^p_{\gamma(L^2(0,t),E_\eta)}\nonumber\\
  & =  \big(1+\|g\|_{L^1(-T,0)}\big)\sup_{\substack{t\in[0,T] \\
 r\in[-t,0]}} \big\|u\mapsto (t+r-u)^{-\alpha} (\phi - \psi)(u)\big\|^p_{\gamma(L^2(0,t+r),E_\eta)}\nonumber\\
  & =  \big(1+\|g\|_{L^1(-T,0)}\big)\sup_{t\in[0,T]} \big\|s\mapsto (t-s)^{-\alpha} (\phi - \psi)(s)\big\|^p_{\gamma(L^2(0,t),E_\eta)}.
\end{align}
Taking expectations and combining \eqref{eqn:B_contraction1} with \eqref{eqn:B_contraction2} and \eqref{eqn:B_contraction3}, there exists $C\geq0$ such that
\begin{equation}
 \label{eqn:B_contraction4}
   \big\|S\diamond\big(B(\cdot,\widetilde{\phi}_\cdot) - B(\cdot,\widetilde{\psi}_\cdot)\big)\big\|_{\subV} \leq C T^{\varepsilon'} \|\phi-\psi\|_{\subV}
\end{equation}
as required.

\mbox{}\\
\emph{Step 4} (Collecting the estimates).

From the estimates \eqref{eqn:initial_estimate}, \eqref{eqn:F_contraction} and \eqref{eqn:B_contraction4} we see that $L_T$ is well defined on $\V$ and that there exist constants $C\geq0$ and $\beta>0$ such that for all $\phi,\psi\in \V$ we have 
\begin{equation}
 \label{eqn:L_contraction}
  \big\|L_T(\phi) - L_T(\psi)\|_{\subV} \leq CT^\beta \|\phi-\psi\|_{\subV}.
\end{equation}
Now recall the definition of $\widehat{\Phi}$ \eqref{eqn:defHat} as the extension of $\Phi$ by $0$ to $(-\infty,T_0]$
\begin{align}
 \big\|S(\cdot)\Phi(0)\big\|_{\subV} &\quad\stackrel{\eqref{eqn:initial_estimate}}{\!\leq}\quad \!\!C\Big( \mathbb{E}\big\|\Phi(0)\big\|_{E_\eta}^p\Big)^{\frac{1}{p}}\nonumber\\
  &\quad\stackrel{\eqref{eqn:KM_ineq}}{\leq}\quad CH\Big( \mathbb{E}\big\|\Phi\big\|_{\curlyB_\eta}^p\Big)^{\frac{1}{p}}\nonumber\\
 \big\|S*F(\cdot,\widehat{\Phi}_\cdot)\big\|_{\subV} &\quad\stackrel{\eqref{eqn:deterministic_estimate_3}}{\!\!\leq}\quad \!\!CT^\beta\Big( \mathbb{E}\big\|F(\cdot,\widehat{\Phi}_\cdot)\big\|_{C([0,T];E_\eta)}^p\Big)^{\frac{1}{p}}\nonumber\\
  &\quad\stackrel{\eqref{eqn:F-linear-growth}}{\leq}\quad C_F C T^\beta \Big(1 + \mathbb{E}\big\|\Phi\big\|_{\curlyB_\eta)}^p\Big)^{\frac{1}{p}}.\nonumber
\end{align}
By \eqref{eqn:Phi_gamma_lipschitz}, $B(\cdot,\widehat{\Phi}_\cdot)$ satisfies \eqref{eqn:Psi_assumption1}, and hence by \eqref{eqn:stochastic_estimate1}
\begin{align}
 \big\|S\diamond B(\cdot,\widehat{\Phi}_\cdot)&\big\|_{\subV} \nonumber\\
&\leq CT^\beta\Big(\sup_{t\in[0,T]}\mathbb{E}\big\|s\mapsto(t-s)^{-\alpha}B(s,\widehat{\Phi}_s)\big\|^p_{\gamma(L^2(0,t;H),\curlyB_\eta)}\Big)^{\frac{1}{p}}\nonumber\\
&\stackrel{\eqref{eqn:B-linear-growth}}{\leq} CT^\beta C^\gamma_B\Big(1+ \sup_{t\in[0,T]}\mathbb{E}\big\|s\mapsto(t-s)^{-\alpha}\widehat{\Phi}_s\big\|^p_{L^2_\gamma(0,t;\curlyB_\eta)}\Big)^{\frac{1}{p}}\nonumber
\end{align}
Combining the above gives
\begin{align}
  \big\|L_T(0)\big\|_{\subV} &\leq C\Big(1 + \big(\mathbb{E}\|\Phi\|^p_{\curlyB_\eta}\big)^{\frac{1}{p}}\nonumber\\
  & + \big(\!\sup_{t\in[0,T]}\mathbb{E}\big\|s\mapsto(t-s)^{-\alpha}\widehat{\Phi}_s\big\|^p_{L^2_\gamma(0,t;\curlyB_\eta)}\big)^{\frac{1}{p}}\Big)\nonumber
\end{align}
which together with \eqref{eqn:L_contraction} yields the result.
\end{proof}

\begin{thm}[Existence and Uniqueness]
\label{thm:6.2}
 Suppose (D1) -- (D6) are satisfied and choose $\alpha \in (0,1/2)$ such that 
\[
 \eta + \theta_B < \alpha - \frac{1}{p}.
\]
There exists a mild solution $U$ in $V^p_{\alpha,\infty}\big([0,T_0]\times\Omega;E_\eta\big)$ of \eqref{eqn:DSDE}. As a mild solution in $V^p_{\alpha,p}\big([0,T]\times\Omega;E_\eta\big)$, this solution $U$ is unique. Moreover there exists a constant $C\geq0$ independent of $\Phi$ such that 
\begin{align}
 \label{eqn:U_linear_growth1}
\|U\|_{V^p_{\alpha,\infty}([0,T_0]\times\Omega;E_\eta)} \leq& C\Big(1 + \big(\mathbb{E}\|\Phi\|^p_{\curlyB_\eta}\big)^{\frac{1}{p}} \nonumber\\
&+ \big(\!\!\sup_{t\in[0,T]}\mathbb{E}\big\|s\mapsto(t-s)^{-\alpha}\widehat{\Phi}_s\big\|^p_{L^2_\gamma(0,t;\curlyB_\eta)}\big)^{\frac{1}{p}}\Big).
\end{align}
\end{thm}

\begin{proof}
 Again, we follow the method used in Theorem 6.3 of \cite{VanN-Weis_1}. By Proposition \ref{prop:6.1} we can find $T\in (0,T_0]$, independent of $\Phi$, such that $C_T < 1/2$. It follows from \eqref{eqn:L_Lipschitz} and the Banach fixed point theorem that $L_T$ has a unique fixed point $U\in \V$. This gives a continuous adapted process $U:[0,T]\times\Omega \to E_\eta$ such that almost surely for all $t\in[0,T]$
\begin{equation}
\label{eqn:U_mild_soln}
 U(t) = S(t)\Phi + S*F(\cdot,\widetilde{U}_\cdot)(t) + S\diamond B(\cdot,\widetilde{U}_\cdot)(t).\nonumber
\end{equation}
Noting that $U = \lim_{n\to\infty}L_T^n(0)$ in $\V$, \eqref{eqn:L_linear_growth} implies the inequality
\begin{align}
 \|U\|_{\subV} \leq& C\Big(1 + \big(\mathbb{E}\|\Phi\|^p_{\curlyB_\eta}\big)^{\frac{1}{p}}\nonumber\\
&+ \big(\!\!\sup_{t\in[0,T]}\mathbb{E}\big\|s\mapsto(t-s)^{-\alpha}\widehat{\Phi}_s\big\|^p_{L^2_\gamma(0,t;\curlyB_\eta)}\big)^{\frac{1}{p}}\Big) \nonumber\\
&+ C_T\|U\|_{\subV},\nonumber
\end{align}
and then $C_T<1/2$ implies 
\begin{align}
\label{eqn:U_linear_growth2}
 \|U\|_{V^p_{\alpha,\infty}([0,T]\times\Omega;E_\eta)} \leq& C\Big(1 + \big(\mathbb{E}\|\Phi\|^p_{\curlyB_\eta}\big)^{\frac{1}{p}}\\
&+ \big(\!\!\sup_{t\in[0,T]}\mathbb{E}\big\|s\mapsto(t-s)^{-\alpha}\widehat{\Phi}_s\big\|^p_{L^2_\gamma(0,t;\curlyB_\eta)}\big)^{\frac{1}{p}}\Big).\nonumber
\end{align}

To find a solution on all of $[0,T_0]$ we show that $U_T$ satisfies (D6) and hence a solution exists on $[T,2T]$. Then by induction we can construct mild solutions on each of the intervals \mbox{$[T,2T], \ldots, [nT,T_0]$} for some $n$. The induced solution $U$ on $[0,T_0]$ is the desired mild solution of \eqref{eqn:DSDE} and is unique. Moreover, by \eqref{eqn:U_linear_growth2} and induction we deduce \eqref{eqn:U_linear_growth1}.

Now $U_T:\Omega\to\curlyB_\eta$ is strongly $\mathscr{F}_T$-measurable, $U(T) \in L^p(\Omega,E_\eta)$ and setting
\begin{equation}
\label{eqn:defHatU}
 \widehat{U}(t) := \left\{ \begin{array}{ll}
			  \Phi(t) & t\in(-\infty,0]\\
			  U(t) & t\in (0,T]\\
			  0 & t\in (T,T_0]
			\end{array}\right.
\end{equation}
and $T_1 = T_0 - T$ we have
\begin{align}
 \sup_{t\in[0,T_1]}\mathbb{E}\big\|s&\mapsto(t-s)^{-\alpha}\widehat{U}_{T+s}\big\|_{\gamma(L^2(0,t;H),\curlyB_\eta)} \nonumber\\
&\!\!\!\stackrel{\eqref{eqn:gamma_curlyB_isomorphism}}{\lesssim}  \!\!\!\sup_{t\in[0,T_1]}\mathbb{E} \int_{-\infty}^0\!\!\!\!\!g(r)\big\|s\mapsto(t-s)^{-\alpha}\widehat{U}(T+s+r)\big\|_{\gamma(L^2(0,t;H),E_\eta)} \ud r\nonumber\\
&\qquad\qquad\qquad + \sup_{t\in[0,T_1]}\mathbb{E} \big\|s\mapsto(t-s)^{-\alpha}\widehat{U}(T+s)\big\|_{\gamma(L^2(0,t;H),E_\eta)}\nonumber\\
& = \sup_{t\in[0,T_1]}\mathbb{E}\big(I_1 + I_2\big)\nonumber
\end{align}
Now $I_2 = 0$ by the definition of $\widehat{U}$ and by partitioning the range $(0,t)$ as above, 
\begin{align}
I_1 &= \int_{-\infty}^0\!\!\!\!\!g(r)\big\|w\mapsto(t\!+\!T\!+\!r\!-\!w)^{-\alpha}\widehat{U}(w)\big\|_{\gamma(L^2(T+r,t+T+r;H),E_\eta)} \ud r\nonumber\\
&\leq \int_{-\infty}^0\!\!\!\!\!g(r)\big\|w\mapsto(t\!+\!T\!+\!r\!-\!w)^{-\alpha}\Big[\mathbbm{1}_{(-\infty,0]}\Phi(w) \nonumber\\
& \qquad\qquad\qquad\qquad\qquad\qquad\qquad + \mathbbm{1}_{(0,T]}(w)U(w)\Big]\big\|_{\gamma(L^2(T+r,t+T+r;H),E_\eta)}\ud r\nonumber\\
&\leq  \int_{-\infty}^0\!\!\!\!\!g(r)\Big[\big\|w\mapsto(t\!+\!T\!+\!r\!-\!w)^{-\alpha}\Phi(w)\big\|_{\gamma(L^2(T+r,(t+T+r)\wedge0;H),E_\eta)}\nonumber\\
& \quad + \big\|w\mapsto(t\!+\!T\!+\!r\!-\!w)^{-\alpha}U(w)\big\|_{\gamma(L^2((t+T+r)\wedge0,(t+T+r)\wedge T;H),E_\eta)}\Big]\ud r\nonumber\\
&\leq \Big[ \int_{-\infty}^0\!\!\!\!\!g(r)\big\|s\mapsto(t-s)^{-\alpha}\widehat{\Phi}_{T+s}(r)\big\|_{\gamma(L^2(0,t;H),E_\eta)}\ud r\nonumber\\
& \quad + \int_{-(t+T)}^0\!\!\!\!\!\!\!\!\!\!\!g(r)\big\|w\mapsto(t\!+\!T\!+\!r\!-\!w)^{-\alpha}U(w)\big\|_{\gamma(L^2((t+T+r)\wedge0,(t+T+r)\wedge T;H),E_\eta)}\ud r \Big]\nonumber\\
&\!\!\!\stackrel{\eqref{eqn:gamma_curlyB_isomorphism}}{\lesssim} \big\|g(r)\big\|_{L^1(-T_0,0)}\Big[\big\|s\mapsto(t-s)^{-\alpha}\widehat{\Phi}_{T+s}\big\|_{\gamma(L^2(0,t;H),\curlyB_\eta)}\nonumber\\
&\quad+ \!\!\!\!\!\!\sup_{r\in[-(t+T),0]}\big\|w\mapsto(t\!+\!T\!+\!r\!-\!w)^{-\alpha}U(w)\big\|_{\gamma(L^2(0,(t+T+r)\wedge T;H),E_\eta)}\Big]\nonumber\\
&\lesssim \big\|s\mapsto(t-s)^{-\alpha}\widehat{\Phi}_{T+s}\big\|_{\gamma(L^2(0,t;H),\curlyB_\eta)} + \big\|s\mapsto(t-s)^{-\alpha}U(s)\big\|_{\gamma(L^2(0,t;h),E_\eta)}\nonumber
\end{align}
so $\sup_{t\in[0,T_1]}\mathbb{E}(I_1)$ is finite as $U\in \V$ and $U_T$ satisfies (D6). The result follows.
\end{proof}

Space and time regularity of the solution follows very quickly from the above, as H\"older regularity of the convolutions was already proved in \cite{VanN-Weis_1} Lemma 3.6 and Proposition 4.2, which up until now have not been used to their full extent. The proof follows exactly the same lines as \cite{VanN-Weis_1} and so is omitted.

\begin{thm}
  Let $E$ be a UMD space with type $\tau\in[1,2]$ and suppose that (D1) - (D6) hold. Choose $\alpha \in (0,1/2)$ such that 
\[
 \eta + \theta_B < \alpha - \frac{1}{p}
\]
Let $\lambda\geq 0$ and $\delta\geq\eta$ satisfy $\lambda+\delta < \min\{1/2-1/p - \theta_B, 1-\theta_F\},$ then there exists $C\geq0$ such that
\[
  \Big( \mathbb{E}\big\|U-S\Phi(0)\|^p_{C^\lambda([0,T_0];E_\delta)}\Big)^{\frac{1}{p}} \leq C\Big(1+\Big(\mathbb{E}\|\Phi\|^p_{\curlyB_\eta}\Big)^{\frac{1}{p}}\Big).
\]

\end{thm}

\section{An Example}
\label{sec:example}

Let $S$ be an open and bounded subset of $\reals^d$ and consider the following perturbed heat equation with memory, equipped with Dirichlet boundary conditions
\begin{align}
\label{eqn:example_DSDE}
  \frac{\partial u}{\partial t}u(t,s) &\quad=\quad  \Delta u(t,s) + f\big(t,s,u_t(s)\big) + \sum_{n\geq1} b_n\big(t,s,u_t(s)\big)\frac{\partial W_n}{\partial t}(t),\nonumber\\
&\qquad\qquad\qquad\qquad\qquad\qquad\qquad\qquad\qquad t\in[0,T],\;s\in S,\\
  u(t,s) &\quad=\quad 0, \quad t\in[0,T],\; s\in\partial S,\nonumber\\
  u_0(t,s) &\quad=\quad \Phi(t,s), \quad t\in(-\infty,0],\; s\in S.\nonumber\\
  u(0,s) &\quad=\quad \Phi(0,s)
\end{align}
Let $p>2$ and $E := L^p(S)$. It is well known that the Dirichlet Laplacian $\Delta_p$ generates a uniformly exponentially stable analytic $C_0$-semigroup on $E$. $(W_n)_{n\geq1}$ is a sequence of independent standard Brownian motions on $\Omega$. We say that $u:[0,T]\times\Omega\times S \to \reals$ is a solution of \eqref{eqn:example_DSDE} if the corresponding functional analytic model \eqref{eqn:DSDE} has a mild solution $U$ with $U(t,\omega)(s) = u(t,\omega,s)$.

We assume the functions $f,b_n : [0,T]\times \Omega\times S\times L^p(-\infty,0] \to \reals$ are jointly measurable and adapted, and that there exist constants $L_f, L_{b_n}\geq0$ such that for all $t\in[0,T]$, $\omega\in\Omega$, $s\in S$ and all functions $\phi,\psi:(-\infty,T]\to \reals$, that are continuous on $[0,T]$ with $\phi_0,\psi_0\in L^p(-\infty,0]$ we have
\begin{equation}
 \label{eqn:example_f_lipshitz}
  |f(t,\omega,s,\phi) - f(t,\omega,s,\psi)| \leq L_f\|\phi_0 - \psi_0\|_{L^p(-\infty,0]}
\end{equation}
and
\begin{equation}
 \label{eqn:example_b_lipshitz}
  |b_n(t,\omega,s,\phi) - b_n(t,\omega,s,\psi)| \leq L_{b_n}\big\|(\phi - \psi)_t(\cdot)\big\|_{L^p(-\infty,0]}.
\end{equation}
where $\sum_{n\geq1}b_n^2 < \infty$. Assume also that 
\begin{equation}
\label{eqn:expl_f_linear_growth}
 \sup_{t\in[0,T],\,\omega\in\Omega}\big\|f(t,\omega,\cdot,0)\big\|_{L^p(S)} < \infty
\end{equation}
 and
\begin{equation}
\label{eqn:expl_b_linear_growth}
 \sup_{\omega\in\Omega}\Bigg\|\Bigg(\int_0^T \sum_{n\geq1} \big|b_n(t,\omega,\cdot,0)\big|^2 \ud\mu(t)\Bigg)^{\frac{1}{2}}\Bigg\|_{L^p(S)} <\infty
\end{equation}
for all finite measures $\mu$ on $[0,T]$.

\begin{thm}
 Assume that the above holds, that $\Phi(0,\cdot)\in L^p(\Omega,E)$ and that the initial history map $t\mapsto\widehat{\Phi}_t$ lies in 
\[
 L^p\big((-\infty,0]\times S; L^2[0,T]\big) \cap L^2\big([0,T];L^p\big((-\infty,0]\times S\big)\big),
\]
i.e.~both the following are finite
\[
 \int_S \int_{-\infty}^0\!\!\Bigg(\int_0^T \!\!|\widehat{\Phi}(r+t,s)|^2 \ud r\Bigg)^{\frac{p}{2}}\!\!\!\! \ud t \ud s, \quad\quad \int_0^T\!\!\Bigg(\int_S\int_{-\infty}^0 \!\!|\widehat{\Phi}(r+t,s)|^p \ud t \ud s\Bigg)^{\frac{2}{p}}\!\!\!\! \ud r.
\]
Then for all $\alpha\in(0,1/2 - 1/p)$ the problem \eqref{eqn:example_DSDE} has a unique mild solution $U\in V^p_{\alpha,p}\big([0,T]\times\Omega;L^p(S)\big)$. 
\end{thm}

\begin{proof}
 We check the conditions of Theorem \ref{thm:6.2}.
\begin{itemize}
 \item[(D1)] As noted above, the Dirichlet Laplacian $\Delta_p$ generates a uniformly exponentially stable analytic $C_0$-semigroup on $E$. Let $H = \ell^2$ with the standard unit basis $(e_n)$, then setting $W_H(t)e_n = W_n(t)$, $W_H$ becomes an $H$-cylindrical Brownian motion.
 \item[(D2)] Take $\eta = \theta_F = \theta_B = 0$.
 \item[(D3)] Take $\curlyB = L^p\big((-\infty,0]\times S\big)\times L^p(S)$.
 \item[(D4)] Define $F:[0,T]\times\Omega\times\curlyB\to E$ by 
\[
 F\big(t,\omega,\phi\big)(s) = f\big(t,\omega,s,\phi(\cdot,s)\big).
\]
Let $\phi, \psi\in \curlyB$, then by \eqref{eqn:example_f_lipshitz} and \eqref{eqn:expl_f_linear_growth} we have (D4).
 \item[(D5)] Define $B:[0,T]\times\Omega\times\curlyB \to \mathcal{L}(H,E)$ by
\[
 \Big(B(t,\omega,\phi)e_n\Big)(s) = b_n\big(t,\omega,s,\phi(\cdot,s)\big).
\]
Let $\phi, \psi: (-\infty,T]\times S\to E$ be continuous on $[0,T]$ and suppose $(t\mapsto\phi_t)$ and $(t\mapsto\psi_t)$ lie in 
\[
  L^p\big((-\infty,0]\times S; L^2(0,T,\mu)\big) \cap L^2\big(0,T,\mu;L^p\big((-\infty,0]\times S\big)\big)                                                                                                                                 
\]
for all finite measures $\mu$ on $[0,T]$. By the $\gamma$-Fubini isomorphism \ref{prop:2.6} we have $(t\mapsto\phi_t)$, $(t\mapsto\psi_t)\in L^2_\gamma(0,T,\mu;\curlyB)$. Now by \eqref{eqn:example_b_lipshitz}, we have 
\begin{align}
 \big\|b_n(\cdot,\omega,s,\phi) - b_n(\cdot,&\omega,s,\psi)\big\|_{L^2(0,T,\mu)}\nonumber\\
  &\leq\quad L_{b_n}\big\|t\mapsto (\phi - \psi)_t(\cdot,s)\big\|_{L^2(0,T,\mu;L^p(-\infty,0]\times\reals)}\nonumber
\end{align}
and so
\begin{align}
 \big\|B(\cdot,\phi) - &B(\cdot,\psi)\big\|^p_{\gamma(L^2(0,T,\mu;H),E)} \nonumber\\
& \simeq_P  \big\|B(\cdot,\phi) - B(\cdot,\psi)\big\|_{L^p(S;L^2(0,T,\mu;H))} \nonumber\\
& =  \int_S \big\| \Big( b_n(\cdot,s,\phi_\cdot(*,s)) - b_n(\cdot,s,\psi_\cdot(*,s))\Big)_{n\geq1}\big\|^p_{L^2(0,T,\mu;\ell^2)}\ud s\nonumber\\
& =  \int_S \Bigg( \sum_{n\geq1} \big\| b_n(\cdot,s,\phi_\cdot(*,s)) - b_n(\cdot,s,\psi_\cdot(*,s)) \big\|^2_{L^2(0,T,\mu)}\Bigg)^{\frac{p}{2}}\ud s\nonumber\\
& \leq \int_S \Bigg( \sum_{n\geq1} L_{b_n}^2 \big\|t\mapsto (\phi-\psi)_t(*,s)\big\|^2_{L^2_\gamma(0,T,\mu;L^p(-\infty,0])}\Bigg)^{\frac{p}{2}}\ud s\nonumber\\
& \leq L^p \int_S \big\|t\mapsto(\phi-\psi)_t(*,s)\big\|^p_{L^2_\gamma(0,T,\mu;L^p(-\infty,0])} \ud s\nonumber\\
& \leq L^p \big\|t\mapsto \phi_t - \psi_t\big\|^p_{L^2_\gamma(0,T,\mu;\curlyB)}\nonumber
\end{align}
where $L:= \big(\sum_{n\geq1} L_{b_n}^2\big)^{\frac{1}{2}}$. This gives us the Lipschitz part of (D5), and the linear growth part follows from \eqref{eqn:expl_b_linear_growth}.
\item[(D6)] By Fubini's Theorem $\Phi(0,\cdot): S\times\Omega\to\reals$ is in $L^p(\Omega,E)$ by assumption and 
\begin{align}
 \big\|r\mapsto (t-r)^{-\alpha}&\widehat{\Phi}_r(\cdot,*)\big\|_{L^2_\gamma(0,t;\curlyB)}\nonumber\\
& \leq \big\|r\mapsto (t-r)^{-\alpha}\widehat{\Phi}_r(\cdot,*)\big\|_{L^2(0,t;\curlyB)} \nonumber\\
& \qquad\qquad + \big\|r\mapsto (t-r)^{-\alpha}\widehat{\Phi}_r(\cdot,*)\big\|_{L^p((-\infty,0]\times S;L^2(0,t))}\nonumber\\
& < \infty.\nonumber
\end{align}
\end{itemize}
The result follows from Theorem \ref{thm:6.2}.
\end{proof}

\section*{Acknowledgements}
We thank Jan van Neerven, Mark Veraar and Sonja Cox for very helpful discussions.

\end{document}